\documentclass[11pt]{amsart}
\usepackage{amssymb}
\usepackage{amsmath}
\usepackage{amscd}
\usepackage{graphicx}
\usepackage{latexsym}

\def\R{\mathbb{R}}
\def\Z{\mathbb{Z}}
\newcommand{\spinc}{\mathrm{Spin}^c}
\newcommand{\CPb}{\overline{\mathbb{CP}}{}^{2}}
\newcommand{\CP}{{\mathbb{CP}}{}^{2}}
\newcommand{\SW}{{\rm SW}}

\def \x {\times}
\def \eu{{\text{e}}}

\newtheorem{theorem}{Theorem}[section]
\newtheorem{proposition}[theorem]{Proposition}
\newtheorem{lemma}[theorem]{Lemma}

\newtheorem{definition}[theorem]{Definition}
\newtheorem{remark}[theorem]{Remark}
\newtheorem{example}[theorem]{Example}

\begin{document}
\title[Broken Lefschetz fibrations and near-symplectic $4$-manifolds]
{Topology of broken Lefschetz fibrations \\ \vspace{0.03in} and near-symplectic $4$-manifolds}
\author[R. \.Inan\c{c} Baykur] {R.\, \.Inan\c{c} Baykur}
\address{Department of Mathematics, Columbia University, New York, NY 10027}
\email{baykur@math.columbia.edu}

\begin{abstract}
The topology of broken Lefschetz fibrations is studied by means of handle decompositions. We consider a slight generalization of round handles, and describe the handle diagrams for all that appear in dimension four. We establish simplified handlebody and monodromy representations for a certain subclass of broken Lefschetz fibrations/pencils, while showing that all near-symplectic closed $4$-manifolds can be supported by these \`a la Auroux, Donaldson, Katzarkov. Various constructions of broken Lefschetz fibrations and a generalization of the symplectic fiber sum operation to the near-symplectic setting are given. Extending the study of Lefschetz fibrations, we detect certain constraints on the symplectic fiber sum operation to result in a $4$-manifold with nontrivial Seiberg-Witten invariant, as well as the self-intersection numbers that sections of broken Lefschetz fibrations can acquire.
\end{abstract}
\maketitle

\setcounter{section}{-1}
\section{Introduction}
In the last decade, symplectic topology has been extensively used to explore the world of smooth $4$-manifolds, where Donaldson's work which provided a description of symplectic $4$-manifolds in terms of Lefschetz fibrations up to blow-ups played a remarkable role. Auroux, Donaldson and Katzarkov extended this result to establish a correspondence between the larger class of near-symplectic $4$-manifolds and an appropriate generalization of Lefschetz fibrations up to blow-ups \cite{ADK}. A detailed topological study of these fibrations, called \emph{broken Lefschetz fibrations} herein, and generalization of various ideas succeeded in the study of symplectic $4$-manifolds to this broader setting are the main themes of our article.

Our goal is to give handlebody descriptions of broken Lefschetz fibrations so to assist with identifying the total spaces of these fibrations, and with calculating smooth invariants. Although we only advert to the Seiberg-Witten invariant in this paper, two other invariants motivate our studies very much. One is the Heegaard-Floer invariant of Ozsv\'ath and Szab\'o which fits in a TQFT and makes use of handle decompositions. Recall that the Heegaard-Floer invariant of any symplectic $4$-manifold was shown to be nontrivial using a decomposition that arise from the underlying Lefschetz pencil structure after Donaldson \cite{OS:symplectic}. The second one is the recently introduced Lagrangian matching invariant of Perutz, associated to broken Lefschetz fibrations \cite{P1, P2}. This is generalized from the Donaldson-Smith invariant defined in the presence of a symplectic Lefschetz pencil \cite{DS}, which was shown to be equivalent to the Seiberg-Witten invariant for pencils of high degree by Usher \cite{U}. Perutz conjectured his invariant to be a smooth invariant and its calculation to be indeed independent of the broken Lefschetz fibration that is chosen on a fixed $4$-manifold. An affirmative answer to this conjecture in turn would require having ways of constructing broken Lefschetz fibrations, preferably as simple as possible in a sense that eases the calculations of the invariant. Our results in this article, which we briefly present below, are planned to be groundwork for future studies in these directions. 

In Section \ref{topology} we describe a generalization of round handle attachments and show that there are exactly two types: classical round handles and a twisted version of them. In dimension four, this is in agreement with the corresponding near-symplectic local models \cite{ADK, Ho2}. The handlebody diagrams for both untwisted and twisted round $1$-handles as well as round $2$-handles are given in Sections \ref{Round1handles} and \ref{Round2handles}. We introduce a subclass of broken Lefschetz fibrations, called \emph{simplified broken Lefschetz fibrations}, that can be effectively described in terms of handlebody diagrams and monodromy representations similar to those of Lefschetz fibrations (cf. Section \ref{simplifiedBLF}). The existence of simplified broken Lefschetz pencils on any $4$-manifold which does not have a negative-definite intersection form is proved in Theorem \ref{simplifiedBLFtheorem}.

In Section \ref{ConstructionsofBLF} we provide several constructions of broken Lefschetz fibrations. We give handlebody diagrams for near-symplectic broken Lefschetz fibrations and pencils on some standard $4$-manifolds, with untwisted and twisted round handles, with connected and disconnected fibers. Also a handlebody description of a broken Lefschetz fibration on the connected sum of the total spaces of two such fibrations is given. Recall that there is a different generalization of Lefschetz fibrations and pencils which are allowed to have nodal singularities and base points with nonstandard orientations, and decorated with the adjective ``achiral''. In Subsection \ref{achiralversusbroken} we give examples of $4$-manifolds which do not admit achiral Lefschetz fibrations or pencils but broken fibrations. We also show how one can turn achiral fibrations into broken ones after blowing-up the $4$-manifold. It follows from the work of Gay and Kirby who proved the existence of broken achiral Lefschetz fibrations on arbitrary closed smooth oriented $4$-manifolds that after blow-ups any such $4$-manifold admits a broken Lefschetz fibration \cite{GK}. Another operation we introduce in this section is the \textit{broken fiber sum} which generalizes the well-known symplectic fiber sum operation \cite{MW, Go} to the near-symplectic setting (Theorem \ref{NSbrokensum}), and whose effect on the Seiberg-Witten invariants is tractable. 

In Section \ref{applications}, we turn our attention to understanding how far certain facts regarding symplectic \linebreak $4$-manifolds and Lefschetz fibrations can be pushed into the near-symplectic geometry. Here we view $4$-manifolds with nontrivial Seiberg-Witten invariants (\SW) as an intermediate class that lies in between near-symplectic and symplectic $4$-manifolds. A question we address is the behavior of near-symplectic $4$-manifolds with nontrivial Seiberg-Witten invariants under the symplectic fiber sum operation. Even though symplectic fiber sum of two symplectic $4$-manifolds is again symplectic and thus has nontrivial \SW, we show that the symplectic fiber sum of near-symplectic $4$-manifolds with nontrivial \SW invariants can result in $4$-manifolds with trivial or nontrivial invariants depending on the choice of fibers that we sum along (Theorem \ref{trivialsum}). In a comparison with symplectic Lefschetz fibrations, we determine the constraints on the self-intersection numbers of sections of broken Lefschetz fibrations, possibly with total spaces which have nontrivial \SW invariants (Theorem \ref{sections}). Lastly, in Proposition \ref{BLFonknotsurgered} we construct near-symplectic broken Lefschetz fibrations on a family of $4$-manifolds which are not symplectic but have nontrivial Seiberg-Witten invariants; namely on knot surgered elliptic surfaces where the knots are nonfibered. 

\vspace{0.3cm}
\noindent \textit{Acknowledgments:} The results contained in this article are from the author's thesis studies. The author would like to thank his adviser Ron Fintushel for his continuous support during the course of this work, David Gay for very valuable e-mail correspondence, Firat Arikan, Cagri Karakurt and Lawrence Roberts for their interest in this study and for their support. Also thanks to Tim Perutz and Andras Stipsicz for commenting on a draft of this paper.

\medskip
\section{Background}
\subsection{Near-symplectic structures. } \label{NSstructures}

Let $\omega$ be a closed $2$-form on an oriented smooth $4$-manifold $X$ such that $\omega^2 \geq 0 $, and $Z_{\omega}$ be the set of points where $\omega \equiv 0$. Then $\omega$ is called a \emph{near-symplectic structure} on $X$ if $\omega^2 > 0$ on $X \setminus Z_{\omega}$ and if it satisfies the following transversality condition at every point $x$ in $Z_{\omega}$: if we use local coordinates on a neighborhood $U$ of $x$ to identify the map $\omega: U \rightarrow \Lambda^2 (T^* U)$ as a smooth map $\omega: \R^4 \to \R^6$, then the linearization $D \, \omega_x : \R^4 \to \R^6$ at $x$ should have rank three ---which is in fact independent of the chosen charts \cite{ADK}. In particular, $Z = Z_{\omega}$ is a smoothly embedded $1$-manifold in $X$. We then call $(X, \omega)$ a near-symplectic $4$-manifold, and $Z$ the \emph{zero locus of $\omega$}.

One of the motivations for studying near-symplectic structures has been the observation that any closed smooth oriented $4$-manifold $X$ with $b^+(X) > 0$ can be equipped with a near-symplectic form, which was known to gauge theory aficionados since early 1980s and a written proof of it was first given by Honda through the analysis of self-dual harmonic $2$-forms (\cite{Ho1}, also see \cite{ADK}). Thus the near-symplectic family is much broader than the symplectic family of $4$-manifolds. For instance, connected sums of symplectic $4$-manifolds can never be symplectic, due to the work of Taubes and the vanishing theorem for \SW invariants. However these manifolds would still have $b^+ >0$ and therefore are near-symplectic.

Using a generalized Moser type argument for harmonic self-dual $2$-forms, Honda showed in \cite{Ho2} that there are exactly two local models around each connected component of $Z_{\omega}$. To make this statement precise, let us consider the following local model: Take $\R^{4}$ with coordinates $(t, x_1,x_2, x_3)$ and consider the $2$-form $\Omega = dt \wedge dQ + *\, (dt \wedge dQ), $ where $Q(x_1, x_2, x_3)= x_1^{2} + x_{2}^{2} - x_{3}^{2}$ and $*$ is the standard Hodge star operator on $\Lambda^{2}\R^{4}$. Restrict $\Omega$ to $\R$ times the unit $3$-ball. Define two orientation preserving affine automorphisms of $\R^4$ by $\sigma_+(t, x_1, x_2, x_3)=( t + 2\pi, x_1, x_2, x_3)$ and $\sigma_- (t, x_1, x_2, x_3)=( t + 2\pi, -x_1, x_2, -x_3)$.  Since both maps preserve $\Omega$, they induce near-symplectic forms $\omega_{\pm}$ on the quotient spaces $N_{\pm} = \R \x D^3 / \sigma_{\pm}$. Honda shows that given any near-symplectic $(X, \omega)$ with zero locus $Z_{\omega}$, there is a Lipschitz self-homeomorphism $\phi$ on $X$ which is identity on $Z_{\omega}$, smooth outside of $Z_{\omega}$ and supported in an arbitrarily small neighborhood of $Z_{\omega}$, such that around each circle in $Z_{\omega}$, the form $\phi^*(\omega)$ agrees with one of the two local near-symplectic models $(N_{\pm}, \omega_{\pm})$. For our purposes, we can always replace the near-symplectic form $\omega$ with such a form $\phi^*(\omega)$. The zero circles which admit neighborhoods $(N_+, \omega_+)$ are called of \emph{even} type, and the others of \emph{odd} type.

\medskip
\subsection{Broken Lefschetz fibrations. }

Let $Z$ be an embedded smooth $1$-manifold, $C$ be a finite set of points in $X \setminus Z$, and $S$ be a compact orientable surface. A smooth map $f: X \to S$ is then called a \emph{broken Lefschetz fibration} if on $X \setminus Z$ it has local models of a Lefschetz fibration with $C$ the critical set, whereas at each $z \in Z$, there are coordinates $(t, x_1, x_2, x_3)$ around $z$ with $t$ a local coordinate on $Z$, in terms of which $f$ is given by $(t, x_1, x_2, x_3) \mapsto (t, x_1^2 + x_2^2 - x_3^2)$. We will call $Z$ the \textit{round singular locus} of $f$, and its image $f(Z)$ the \textit{round singular image}. A \emph{broken Lefschetz pencil} is defined similarly for $S = S^2$, by assuming that there is also a finite set of points $B$ in $X \setminus (Z \cup C)$ and the map $f$ has local models of a Lefschetz pencil instead with $C$ the critical set and $B$ the base locus. 

Broken Lefschetz fibrations and pencils were first introduced by Auroux, Donaldson, and Katzarkov in \cite{ADK} under the name ``singular Lefschetz fibrations'', where it was shown that they are to near-symplectic $4$-manifolds what Lefschetz fibrations are to symplectic $4$-manifolds:

\begin{theorem} [Auroux, Donaldson, Katzarkov \cite{ADK}] \label{ADKtheorem}
Suppose $\Gamma$ is a smooth $1$--dimensional submanifold of a compact oriented $4$--manifold $X$. Then the following two conditions are equivalent:
\begin{itemize}
\item There is a near-symplectic form $\omega$ on $X$,  with $Z_{\omega} = \Gamma$,
\item There is a broken Lefschetz pencil $f$ on $X$ which has round singularities along $\Gamma$, with the property  that there is a class $h\in H^{2}(X)$ such that $h(\Sigma)>0$ for every fiber component $\Sigma$. 
\end{itemize}
Moreover, the implications in each direction can be obtained in a compatible way. That is, given a near-symplectic form $\omega$, a corresponding broken Lefschetz pencil (BLP) can be obtained so that all the fibers are symplectic on the complement of the singular locus. Conversely, given a broken Lefschetz fibration (BLF) satisfying the indicated cohomological condition, one obtains a deformation class of near-symplectic forms which make the regular fibers symplectic.
\end{theorem}

Blowing-up the base locus of a broken Lefschetz pencil we get a broken Lefschetz fibration. When the BLF/BLP supports a near-symplectic structure, these blow-ups/downs are understood to be made symplectically. If we have in hand a BLF over a Riemann surface $S$ that satisfies the cohomological condition in the statement of the theorem, then we can construct compatible near-symplectic forms with respect to which a chosen set of sections are symplectic \cite{ADK}. From now on we will refer to such a fibration $f$ on $X$ as \emph{a near-symplectic broken Lefschetz fibration}, and say that the pair $(X, f)$ is \emph{near-symplectic}. Implicit in this notation is that the near-symplectic form on $X$ is chosen from the unique deformation class of near-symplectic forms compatible with $f$ obtained via Theorem \ref{ADKtheorem}.

\begin{example} \label{StandardNSexample}  \label{StandardBLFexample}  \rm
Let $M^3$ be a closed $3$-manifold and $f: M \to S^1$ be a circle valued Morse function with only index $1$ and $2$ critical points. Then the $4$-manifold $X = S^1 \x M$ can be equipped with a near-symplectic structure. To see this, first note that by a theorem of Calabi there exists a metric $g$ on $M$ which makes $df$ harmonic. Parametrize the $S^1$ component by $t$, and consider the form $\omega = dt \wedge df + * (dt \wedge df)$, where the Hodge star operation is defined with respect to the product of the standard metric on $S^1$ and $g$ on $M$. Thus $\omega^2 \geq 0$ and $\omega$ vanishes precisely on $Z = S^1 \x Crit(f)$. Finally using local charts one can see that $\omega$ vanishes transversally at every point on $Z$, and also that all circles of $Z_{\omega}$ are of even type. If we consider a $3$-manifold $M_K$ obtained from $S^3$ after a $0$-surgery on an arbitrary knot $K$, then it comes with a circle valued Morse function. (Note that $\Z \cong H^1(M_K ; \Z) \cong [M_K, S^1]$.) Then $\omega$ defined as above yields a \textit{symplectic} form on $X= S^1 \x M_K$ if and only if $K$ is fibered so that $f$ can be assumed to have no critical points; i.e when $Z = \emptyset$. 

Now let $X= S^1 \x M_K$ for some nonfibered $K$. For simplicity, assume that the map $f: M_K \to S^1$ above is injective on its critical points. Then the preimage of any regular value of $f$ is a Seifert surface of $K$ capped off with a disk, i.e a closed orientable surface. While passing an index $k$ critical point ($k=1,2$), a $k$-handle is attached to get one Seifert surface from another. It follows that $f: M_K \to S^1$ is a `fibration-like' map, where the genera of fibers are increased or decreased by one at every critical point, depending on $k=1$ or $k=2$, respectively. When crossed with $S^1$, this yields a broken fibration $id \x f: X \to T^2$. The base torus $T^2 = S^1 \x S^1$ can be parameterized by $(t, s)$ where $t$ traces the outer circle factor and $s$ traces the base circle of $f$. The monodromy of this fibration is trivial in the $t$ direction and is prescribed by the knot monodromy in the $s$ direction. 
\end{example}

A BLF over a Riemann surface can be split into Lefschetz fibrations over surfaces with boundaries, and fibered cobordisms between them relating the surface fibrations over the boundary circles. Round singularities of a BLF are contained in these cobordisms, which herein will be called \emph{round cobordisms}. The local models around each round singular circle implies that these cobordisms are given by fiberwise handle attachments, all with the same index (either $1$ or $2$). Roughly speaking, such cobordisms with $1$-handle attachments increase the genus of a fiber component, or connect two different fiber components, whereas cobordisms with $2$-handle attachments either decrease the genus or disconnect a fiber component. We study these cobordisms more rigorously in the next section. In \cite{ADK} it was shown that for any given near-symplectic form $\omega$ on $X$, a compatible broken Lefschetz fibration  $f:X \# b \, \CPb \to S^2$, where $b$ is the number of base points, can be arranged in the following way: The base $S^2$ breaks into three pieces $D_l \cup A \cup D_h$, where $A$ is an annular neighborhood of the equator of the base $S^2$ which does not contain the image of any Lefschetz critical point, $D_l$ and $D_h$ are disks, so that (i) On $X_l = f^{-1}(D_l)$ and $X_h = f^{-1}(D_h)$ we have genuine Lefschetz fibrations; and (ii) The cobordism $W = f^{-1}(A)$ is given by only fiberwise $1$-handle attachments if one travels from the $X_l$ side to $X_h$ side. We call these kind of broken Lefschetz fibrations/pencils \emph{directed}, $X_l$ the \emph{lower side} and $X_h$ the \emph{higher side}.

\medskip
\subsection{Seiberg-Witten invariants. }
\label{subsec:SW}

We now review the basics of Seiberg-Witten invariant. The \emph{Seiberg-Witten invariant} of a smooth closed oriented $4$-manifold $X$\/ is an integer valued function which is defined on the set of Spin$^{c}$ structures on $X$. If we assume that $H_1(X;\Z)$ has no 2-torsion, then there is a one-to-one correspondence between the set of Spin$^{c}$ structures on $X$\/ and the set of characteristic elements of $H^2(X;\Z)$ as follows: To each Spin$^c$ structure $\mathfrak{s}$ on $X$\/ corresponds a bundle of positive spinors $W^+_{\mathfrak{s}}$ over $X$. Let $c(\mathfrak{s}) = c_1(W^+_{\mathfrak{s}}) \in H^2(X;\Z)$. Then each $c(\mathfrak{s})$ is a characteristic element of $H^2(X;\Z)$; \, i.e. $c_1(W^+_{\mathfrak{s}})$ reduces mod~2 to $w_2(X)$.

In this setup we can view the \emph{Seiberg-Witten invariant}\/ as an integer valued function
\begin{equation*}
\SW_X: \{ k\in H_2(X ; \Z) \mid {\rm PD}(k) \equiv w_2(X) \pmod2 \}
\longrightarrow \Z ,
\end{equation*}
where ${\rm PD}(k)$ denotes the Poincar\'{e} dual of $k$.  The Seiberg-Witten invariant $\SW_X$ is a diffeomorphism invariant when $b^+(X)>1$. Its overall sign depends on our choice of an orientation of \linebreak $H^0(X; \R)\otimes\det H_+^2(X;\R)\otimes \det H^1(X;\R).$

If $\SW_X(\beta)\neq 0$, then we call $\beta$\/ (and its Poincar\'{e} dual ${\rm PD}(\beta)\in H^2(X;\Z)$) a \emph{basic class}\/ of $X$. The canonical class $K_X=-c_1(X, \omega)$ of a symplectic $4$-manifold $(X, \omega)$\/ is a basic class when $b^+(X)>1$ with $\SW_{X}(K_X)=1$. It can be shown that, if $\beta$ is a basic class, then so is $-\beta$ with
\begin{equation*}
\SW_X(-\beta)=(-1)^{(\eu(X) +\sigma(X))/4}\,\SW_X(\beta) ,
\end{equation*}
where $\eu(X)$ is the Euler characteristic and $\sigma(X)$ is the signature of $X$.  We say that $X$\/ is of \emph{simple type}\/ if every basic class $\beta$\/ of $X$\/ satisfies
\begin{equation*}
\beta^2 = 2\eu(X) + 3\sigma(X).
\end{equation*}
It was shown in \cite{Ta2} that symplectic 4-manifolds with $b_2^+>1$ are of simple type. Let $\Sigma\subset X$\/ be an embedded surface of genus $g(\Sigma)>0$ with $[\Sigma]^2 \geq 0$. If $\beta$\/ is a basic class of $X$, we have the following \emph{adjunction inequality}\/ (cf.\ \cite{OS:adj ineq}):
\begin{equation} \label{adjineq}
-\chi(\Sigma) = 2g(\Sigma)-2 \geq [\Sigma]^2 + |\beta\cdot[\Sigma] |.
\end{equation}

When $b^+(X) = 1$, the Seiberg-Witten invariant $\SW_{X,H}(K)\in \mathbf{Z}$ is defined for
every positively oriented element $H\in H^2_+(X;\mathbf{R})$ and every element $A \in \mathcal{C}(X)$ such that $A\cdot H\neq 0$. We say that $H$ determines a \emph{chamber}. It is known that if $\SW_{X,H}( \cdot)\neq 0$ for some $H\in H^2_+(X;\mathbf{R})$, then $d(A)\ge 0$. The \emph{wall-crossing formula} prescribes the dependence of $\SW_{X,H}(A)$ on the choice of the chamber (that of $H$): if $H, H' \in H^2_+(X;\mathbf{R})$ and $A\in \mathcal{C}(X)$ satisfy $H\cdot H'>0$ and $d(A)\ge 0$, then
\begin{multline*}
\SW_{X,H'}(A)=\SW_{X,H}(A)\\
+
\begin{cases}
0&\text{if $A\cdot H$ and $A \cdot H'$ have the same sign,}\\
(-1)^{\frac{1}{2}d(A)}&\text{if $A\cdot H>0$ and $A\cdot H'<0$,}\\
(-1)^{1+\frac{1}{2}d(A)}&\text{if $A\cdot H<0$ and $A\cdot H'>0$}.
\end{cases}
\end{multline*}

In the presence of a near-symplectic structure $\omega$ on $X$ with $b^+(X)=1$, we will always consider the Seiberg-Witten invariant of $(X, \omega)$ computed in the chamber of $\omega$.

\medskip
\section{Topology of broken Lefschetz fibrations} \label{topology}

Handlebody diagrams of Lefschetz fibrations over $S^2$ are easy to depict and proved to be useful in the study of smooth $4$-manifolds. The reader is advised to turn to \cite{GS} for the details of this \textit{by now} classical theory and its several applications. In this section, we extend these techniques to the study of broken Lefschetz fibrations. For this purpose, we first discuss \emph{round handles} that arise naturally in the context of $4$-dimensional BLFs thoroughly.

In full generality, we are interested in attaching handles in a parameterized way as we now explain: Let $n \geq 4$ and regard $S^1 \x D^{n-1}$ as the total space of a $D^{n-1}$ bundle over $S^1$ defined by the projection map onto the circle component. Fibers can be thought as $(n-1)$-dimensional $k$-handles $D^k \x D^{n-1-k}$ which we would like to attach so that globally their attachments respect the bundle structure. For $0 < k < n-1$ this requires a choice of splitting the trivial $D^{n-1}$ bundle over $S^1$ into $D^k$ and $D^{n-1-k}$ bundles over $S^1$, which would descend from a splitting of the trivial $\R^{n-1}$ bundle over $S^1$ into rank $k$ and rank $n-1-k$ vector bundles over $S^1$. The latter are classified by homotopy classes of mappings from $S^1$ into the Grassmannian $\textbf{G}(n-1, k)$. Since $\pi_1(\textbf{G}(n-1, k)) = \Z_2$ for $0<k < n-1$ and $n -1 \geq 3$, there are two possible splittings up to isotopy. These can be realized using the two orientation preserving self-diffeomorphisms of $\R^k \x \R^{n-1-k}$, where one is the identity map, and the other one is defined by $(x_1, x_2, \ldots, x_{n-2}, x_{n-1}) \mapsto (-x_1, x_2, \ldots, x_{n-2}, -x_{n-1})$. Restricted to $D^k \x D^{n-1-k}$ each splitting specifies an $(n-1)$-dimensional $k$-handle structure on all $D^{n-1}$ fibers of the initial bundle $S^1 \x D^{n-1} \to S^1$ ---simply by specifying the core (and thus the cocore) on each fiber. The boundary restriction on the first component gives an $S^{k-1} \x D^{n-1-k}$ subbundle. The total space $L$ of the last bundle is a submanifold of $S^1 \x D^{n-1}$. Hence for $0 < k < n-1$, we define an \emph{$n$-dimensional general round $k$-handle} as a copy of $S^1 \x D^{n-1}$, attached to the boundary of an $n$-dimensional manifold $X$ by an embedding of $L \hookrightarrow \partial X$. 

The first comprehensive study of round handles is due to Asimov \cite{A}, and more on $4$-dimensional round $1$-handles can be found in \cite{GK}. However, both articles assume a restriction on the way these handles are attached: they only deal with round $k$-handles attached along $S^1 \x S^{k-1} \x D^{n-1-k}$, which in our definition corresponds to the trivial splitting of the $D^{n-1}$ bundle. We will refer to these as \emph{classical round handles}. Our second type of round handle attachment arises from the latter model where $L$ is a $\Z_2$ quotient of $S^1 \x S^{k-1} \x D^{n-1-k}$. We call the round handles attached in the classical way \emph{even} or \emph{untwisted round handles}, and the others \emph{odd} or \emph{twisted round handles} ---corresponding in dimension four to the even and odd local models around circle components of the zero loci of near-symplectic forms as discussed in Subsection \ref{NSstructures}. One then easily sees that:

\begin{lemma}
For $n \geq 4$ and $0 < k < n-1$, a \emph{general} $n$-dimensional round $k$-handle attachment is given by a $k$-handle attachment followed by a $(k+1)$-handle attachment that goes over the $k$-handle geometrically twice, algebraically zero times if it is an untwisted round handle and twice if it is a twisted round handle.
\end{lemma}

The ones which interest us in this paper are $4$-dimensional $1$- and $2$- round handles. From the very definition of broken Lefschetz fibrations we conclude that:

\begin{lemma}
A round cobordism with a connected round locus and embedded round image in a broken Lefschetz fibration is given either by a twisted or untwisted round $1$-handle (dually round $2$-handle) attachment.
\end{lemma}

\noindent After a small perturbation of the BLF we can decompose any round cobordism into round cobordisms with connected round loci, so this lemma indeed tells that any round cobordism appearing in a BLF can be realized as a sequence of round handle attachments. Conversely, if two surface fibrations are related through a round handle attachment, the fibrations on the two ends of such a cobordism $W$ uniquely extend to a broken fibration on $W$ over $S^1 \x I$, with only one round singularity given by the centers of the cores of fiberwise attached $1$-handles (or dually by the centers of the cores of fiberwise attached $2$-handles), which make up the round handle.

Let us describe the attachments in the twisted case more explicitly. The attachment of a twisted round $1$-handle is made along the boundary of a $D^1$ subbundle that traces a Mobius band. This is topologically the $D^2$ neighborhood of a circle in $S^1 \x D^3$ covering the base $S^1$ twice. If we restrict our attention to the $D^1$ bundle (parameterized by $x_1$) over $S^1$, both untwisted	and twisted round $1$-handles can be seen to have attaching regions given by the restriction of this bundle to its boundary (which gives a bi-section of the $D^1$ bundle) times the complementary $D^2$ bundle. Then the twisted and untwisted cases correspond to this bi-section having one component or two components, respectively. Similarly, a twisted round $2$-handle is attached along a collar neighborhood of a Klein Bottle, whereas in the untwisted case we would be gluing along a collar neighborhood of a torus.

\begin{remark} \rm
Recall that a \emph{fold type singularity} of a map from an $n$-dimensional manifold to a surface is locally modeled by $(x, t) \mapsto (x_1^2 + \ldots + x_k^2 - x_{k+1}^2 - \ldots - x_{n-1}^2, t)$ for some $1 \leq k \leq n-1$, where $(x, t) \in \R^{n-1} \x \R$. Since an $n$-dimensional general round $k$-handle naturally admits a map with fold singularities parameterized along a circle, the above lemma can be generalized to maps with such fold singularities in any dimension.
\end{remark}

\medskip
\subsection{Round $1$-handles. } \label{Round1handles}

Regarding the circle factor of an untwisted round $1$-handle $S^1 \x D^1 \x D^2$ as the union of a $0$-handle and a $1$-handle, we can express an untwisted round $1$-handle as the union of a $4$-dimensional $1$-handle $H_1$ and a $2$-handle $H_2$. This handle decomposition can be seen simply by taking the product of the annulus $S^1 \x D^1$ with $D^2$ so to conclude that $H_2$ goes over $H_1$ geometrically twice but algebraically zero times. In the same way, we can realize a twisted round $1$-handle as the union of a $1$-handle $H_1$ and a $2$-handle $H_2$, too. However the underlying splitting this time implies that $H_2$ goes over $H_1$ both geometrically and algebraically twice.

We are ready to discuss the corresponding Kirby diagrams. Recall that our aim is to study the round handle attachments to boundaries of (broken) Lefschetz fibrations. Let $F$ denote the $2$-handle corresponding to the regular fiber. Both in untwisted and twisted cases, the $2$-handle $H_2$ of the round $1$-handle links $F$ geometrically and algebraically twice and can attain any framing $k$. Both `ends' of the $H_2$ are allowed to go through any one of the $1$-handles of the fiber before completely wrapping around $F$ \textbf{once}. In addition, these two ends might twist around each other as in Figure \ref{Round1attachments}. (Caution! The ``twisting'' discussed in \cite{ADK} is not this one; what corresponds to it is the framing $k$.) The difference between untwisted and twisted cases only show-up in the way $H_2$ goes through $H_1$ as demonstrated in the Figure \ref{Round1attachments}.

\begin{figure} \label{Round1attachments}
\begin{center}
\includegraphics{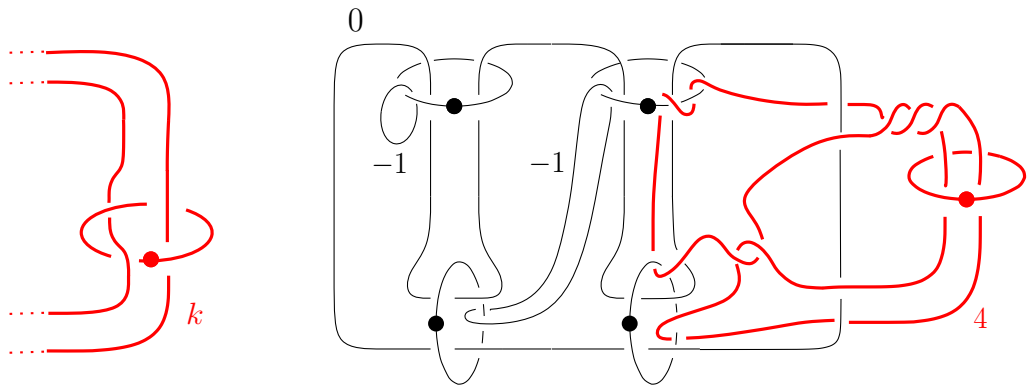}
\caption{\small An arbitrary twisted round $1$-handle (left), and an untwisted round $1$-handle attachment to a genus two Lefschetz fibration over a disk (right).}
\end{center}
\end{figure}

\medskip
\subsection{Round $2$-handles. } \label{Round2handles}
The handle decomposition of round $2$-handles is analogous to that of round $1$-handles. Regarding the circle factor of an untwisted round $2$-handle $S^1 \x D^2 \x D^1$ as the union of a $0$-handle and a $1$-handle, this time we can express an untwisted round $2$-handle as the union of a $4$-dimensional $2$-handle $H'_2$ and a $3$-handle $H'_3$. For a twisted round $2$-handle we get a similar decomposition. However the splittings once again imply the difference: the $3$-handle goes over the $2$-handle geometrically twice and algebraically zero times in the untwisted case, and both geometrically and algebraically twice in the twisted case. One can conclude this from the previous subsection as well, since a round $2$-handle turned upside down is a round $1$-handle.

We are now ready to discuss the corresponding Kirby diagrams for attaching round $2$-handles to Lefschetz fibered $4$-manifolds with boundary. The round $2$-handle attachment to a surface fibration over a circle that bounds a Lefschetz fibration is realized as a fiberwise $2$-handle attachment. The attaching circle of the $2$-handle $H'_2$ of a round $2$-handle is then a simple closed curve $\gamma$ on a regular fiber, which is preserved under the monodromy of this fibration up to isotopy. Since this attachment comes from a fiberwise handle attachment, $H'_2$ should have fiber framing zero. As usual, we do not draw the $3$-handle $H'_3$ of the round $2$-handle, which is forced to be attached in a way that it completes the fiberwise $2$-handle attachments. The difference between the untwisted and twisted cases is then somewhat implicit; it is distinguished by the two possible ways that the curve $\gamma$ is mapped onto itself under a self-diffeomorphism of the fiber determined by the monodromy. If $\gamma$ is mapped onto itself with the same orientation, then we have an untwisted round $2$-handle, and a twisted round $2$-handle if the orientation of $\gamma$ is reversed. The reader might want to refer to the relevant monodromy discussion after the proof of Theorem \ref{simplifiedBLFtheorem}.

The upshot of using round $2$-handles is that one can depict any Lefschetz fibration over a disk together with a round $2$-handle attachment via Kirby diagrams explicitly. One first draws the Lefschetz $2$-handles following the monodromy data on a regular diagram of $D^2 \x \Sigma_g$ (where $g$ is the genus of the fibration) with fiber framings $-1$, then attaches $H'_2$ with fiber framing $0$ and includes an extra $3$-handle. We draw the Kirby diagram with standard $1$-handles so to match the fiber framings with the blackboard framings, which can then carefully be changed to the dotted notation if needed. Importantly, it suffices to study only these type of diagrams when dealing with BLFs on near-symplectic $4$-manifolds, as we will prove in the next subsection.

\begin{figure} [ht] \label{Round2attachments}
\begin{center}
\includegraphics{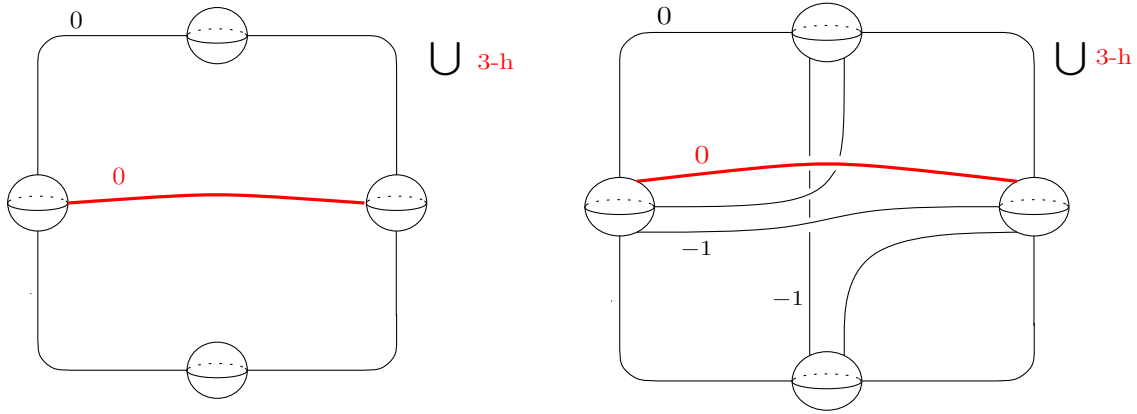}
\caption{\small Left: an untwisted round $2$-handle attachment to $D^2 \x T^2$. Right: a twisted round $2$-handle attachment to an elliptic Lefschetz fibration over a disk with two Lefschetz singularities. Red handles make up the round $2$-handle.}
\end{center}
\end{figure}

To illustrate what we have stated above, let us look at the two simple examples in Figure \ref{Round2attachments}. In the first example the round $2$-handle is attached to a trivial fibration, so $\gamma$ is certainly mapped onto itself with the same orientation. Therefore it is an untwisted round $2$-handle. For the second one, we can express the self-diffeomorphism of the $2$-torus fiber induced by the monodromy $\mu$ by the matrix:
\vspace{-0.3in}
\begin{center}
\begin{math} \bordermatrix{& & &\cr
&-1 &2\cr
&0 &-1\cr}
\end{math}
\end{center}
\vspace{0.1in}
\noindent and the curve $\gamma$ by the matrix $[1 \, \, 0]^{\, T}$. Thus $\mu$ maps $\gamma$ to $- \gamma$, and this yields a twisted round $2$-handle attachment. Both of these examples will be revisited later.

\newpage
\subsection{Simplified broken Lefschetz fibrations} \label{simplifiedBLF}

\begin{definition}
A \emph{simplified broken Lefschetz fibration} on a closed $4$-manifold $X$ is a broken Lefschetz fibration over $S^2$ with only one round singularity and with all critical points on the higher side. A \emph{simplified broken Lefschetz pencil} is a broken Lefschetz pencil that yields a simplified broken Lefschetz fibration on $\tilde{X}$ obtained by blowing-up the base points of the pencil on $X$.
\end{definition}

\noindent Since the total space of the fibration is connected, the ``higher side'' always consists of connected fibers. The fibers on the lower side have lower genus whenever the fibers are connected, while in general the term refers to the direction of the fibration. Simplified BLFs can be depicted efficiently using the handlebody diagrams described in Subsection \ref{Round2handles}. Examples are given in the next subsection.

We will need the following lemma:

\begin{lemma} \label{Push}
If $X$ admits a directed broken Lefschetz fibration over $S^2$, then it can be replaced by a new broken Lefschetz fibration over $S^2$, where all Lefschetz singularities are contained in the higher side.
\end{lemma}

\begin{proof} To begin with, we can perturb the directed fibration so to guarantee that it is injective on the circles of the round locus. Thus the fibration can be split into a Lefschetz fibration over a disk (the lower side), to which we consecutively attach round $1$-handles, and then we close the fibration by another Lefschetz fibration over a disk (the higher side).

To simplify our discussion, for the time being assume that the fibers are all connected, so there is the lower genus side $X_l$ with regular fiber $F_l$, the round handle cobordism $W$, and the higher genus side $X_h$ with regular fiber $F_h$. Let the genus of the regular fibers in the lower side be $g$. The standard handlebody decomposition of $X_l$ consists of a $0$-handle, $2 g$ $1$-handles and some $2$-handles one of which corresponds to the fiber, and the rest to the Lefschetz handles in $X_l$. By our assumption, $W$ is composed of ordered round $1$-handle cobordisms $W_1 \cup W_2 \cup \cdots \cup W_k$, where $k$ is the number of circle components in the round locus. Let us denote the lower side boundary of $W_i$ by $\partial_- W_i$ and the higher side by $\partial_+ W_i$.

Consider $X_l \cup W_1$, which is obtained by adding a round $1$-handle $R_1$ composed of a $1$-handle $H_1$ and a $2$-handle $H_2$. The $\partial (X_l \cup W_1) = \partial_+ W_1 = \partial_- W_2$ is the total space of a genus $g+1$ surface bundle over a circle. We can make sure that the vanishing cycles of the Lefschetz $2$-handles in $X_l$ sit on the fibers of the genus $g$ fibration on $\partial X_l$. Moreover, we can assume that the bi-section which is the attaching region of $R_1$ misses these vanishing cycles. This means that $H_1$ and $H_2$ do not link with any one of the Lefschetz $2$-handles in $X_l$ but only with the $2$-handle corresponding to the fiber and possibly with some of the $1$-handles corresponding to the genera of the fiber. We can replace the handlebody prescribed by the BLF on $X_l \cup W_1$ by another one where first $H_1$ and $H_2$ are attached to the standard diagram of $D^2 \x F_l$, and the Lefschetz $2$-handles are attached afterwards. Having modified the diagram this way, now we can assume that the Lefschetz $2$-handles are attached to $\partial (X_l \cup W_1)$, which can be pulled to $\partial_{-} W_2$ via the fiber preserving diffeomorphism between $\partial_{+} W_1$ and $\partial_{-} W_2$. The fiber framings of these $2$-handles remain the same, and therefore they are still Lefschetz.

Inductively, one slides the Lefschetz $2$-handles so to have them attached to $\partial (X_l \cup W_1 \cup W_2 \cup \cdots \cup W_k) = \partial (X_l \cup W) = - \partial X_h$. Higher side $X_h$ together with these $2$-handles is equipped with a new Lefschetz fibration of genus $g+k$ (which is the same as the genus of $F_h$) over a disk. Hence we obtain a new handlebody decomposition which describes a new BLF on $X$, with all the Lefschetz singularities contained in the new higher side. The reader can verify that a similar line of arguments work when $X_l$ has disconnected fibers.
\end{proof}

Given a near-symplectic form on a closed $4$-manifold $X$, Perutz \cite{P3} and Taubes \cite{T3} independently showed that one can obtain a cohomologous near-symplectic form on $X$ with a connected round locus. The meat of the next theorem is this observation and the Theorem \ref{ADKtheorem}.

\begin{theorem} \label{simplifiedBLFtheorem}
On any closed near-symplectic 4-manifold $(X, \omega)$, possibly after replacing $\omega$ with a cohomologous near-symplectic form $\omega'$, one can find a compatible simplified broken Lefschetz pencil.
\end{theorem}

\begin{proof} If necessary, first replace $\omega$ by a cohomologous form $\omega'$ with connected vanishing locus. Theorem \ref{ADKtheorem} shows that there is a broken Lefschetz pencil compatible with this near-symplectic form, so it should have only one round handle singularity. Symplectically blow-up the base points to obtain a near-symplectic BLF on the blow-up $\tilde{X}$ of $X$. Apply the above lemma to get a simplified Lefschetz fibration on $\tilde{X}$, which also supports the near-symplectic structure since the fibers are unchanged and still symplectic under the modification described in the proof of Lemma \ref{Push}.

The exceptional spheres appear as $2$-handles linked to the higher genus fiber component, all with framing $-1$, and not linking to each other or to any other handle. The modification in Lemma \ref{Push} will not involve these handles; their linkings and framings will remain the same. Since they represent the exceptional spheres, we can symplectically blow them down to obtain a new Lefschetz pencil on $X$, with the desired properties.
\end{proof}

\noindent In fact we can slightly strengthen the choice of $\omega'$ in the proof and make it ``near-symplectically cobordant'' to the original form as in \cite{P3}.

It is no surprise that the monodromy representations of these fibrations are also simpler than usual. Here we include a brief digression on this topic. The reader unfamiliar with this topic can turn to \cite{GS} for the essentials. Let $Map_{\gamma}(F_g)$ be the subgroup of $Map(F_g)$ that consists of elements which fix the embedded curve $\gamma$, up to isotopy. Then there is a natural homomorphism $\phi_{\gamma}$ from $Map_{\gamma}(F_g)$ to $Map(F_{g-1})$ or to $Map(F_{g_1}) \times Map (F_{g_2})$ depending on whether $\gamma$ is nonseparating or separating $F_g$ into two closed oriented surfaces of genera $g_1$ and $g_2$. Define $S_g$ to be the set of pairs $(\mu, \gamma)$ such that $\mu \in Map_{\gamma}(F_g)$ and $\mu \in Ker \, (\phi_{\gamma})$. Recall that when the fiber genus is at least two, fiber-preserving gluing maps are determined uniquely up to isotopy. Hence, given any tuple $(\mu, \gamma) \in S= \bigcup_{g \geq 3} S_g$, we can construct a unique simplified BLF unless $\gamma$ is separating and there is a $g_i \leq 1$. Otherwise, one needs to include the data regarding the gluing of the low genus pieces carrying genus $0$ or genus $1$ fibrations.

If the fibers are connected, the map $\phi_{\gamma}: Map(F_g) \to Map(F_{g-1})$ above factors as 
\[ \psi_{\gamma}:  Map(F_g) \to Map(F_g \setminus N) \ \ \text{and} \ \ \ \varphi_{\gamma}: Map(F_g \setminus N) \to Map(F_{g-1}), \] 
where $N$ is an open tubular neighborhood of $\gamma$ away from the other vanishing cycles. (The middle group does not need to fix the boundaries.) The map $\psi$ has kernel isomorphic to $\Z$ ---the framing of the $2$-handle of a round $1$-handle. When we have a simplified BLF, the kernel of $\varphi$ is isomorphic to the braid group on $F_{g-1}$ with $2$-strands, by definition. This gives an idea about the cardinality of $S$, and in turn about the cardinality of the family of BLFs over $S^2$.

\begin{remark} \rm
Any closed oriented $4$-manifold whose intersection form is not negative-definite admits a simplified broken Lefschetz fibration after blow-ups. These split into two pieces: a symplectic Lefschetz fibration over a disk, and a near-symplectic broken fibration over a disk with only connected round singular locus. Let us restrict our attention to the latter piece, and assume for simplicity that we have connected fibers of genus $g$. Relative Seiberg-Witten-Floer, Heegaard-Floer, and Lagrangian matching invariants of the piece $D^2 \x F_g$ all take values in $H_*(Sym^n(F_g))\,$ for appropriate choices of the spin$^c$ (the degree of which on $F_g$ together with the genus $g$ determines $n$). The round $1$-handle attachment induces a map from these groups to the Floer homology groups of the fibered $3$-manifold separating the two pieces. Hence, these round handle attachments together with the monodromy of the higher genus side determines the computation of any one of these invariants, and for instance it can tell a lot about when the invariants vanish. We will address this problem elsewhere. 
\end{remark}


\medskip
\section{Constructions of Broken Lefschetz fibrations} \label{ConstructionsofBLF}

We start with several examples of simplified broken Lefschetz fibrations. The examples are chosen to span various types of fibrations; with untwisted round locus, twisted round locus, connected fibers, disconnected fibers, and those which do not support any near-symplectic structure. The near-symplectic examples we present here are used later in our paper.

\begin{example} \label{example1} \rm 

The Figure \ref{f:example1} describes a near-symplectic BLF on $S^2 \x \Sigma_g \, \# S^1 \x S^3$, which is composed of a trivial $\Sigma_{g+1}$ fibration on the higher side, a trivial $\Sigma_g$ fibration on the lower side, and an untwisted round $1$-handle cobordism in between. We call this fibration the \emph{step fibration} for genus $g$. To identify the total space, first use the $0$-framed $2$-handle of the round $2$-handle to separate the $2$-handle corresponding to the fiber. Then eliminate the obvious canceling pair, and note that the remaining $1$-handle together with the $3$-handle of the round $2$-handle describes an $S^1 \x S^3$ summand. As the rest of the diagram gives $S^2 \x \Sigma_g$, we see that the total space is as claimed.

\begin{figure} [ht]
\begin{center}
\includegraphics{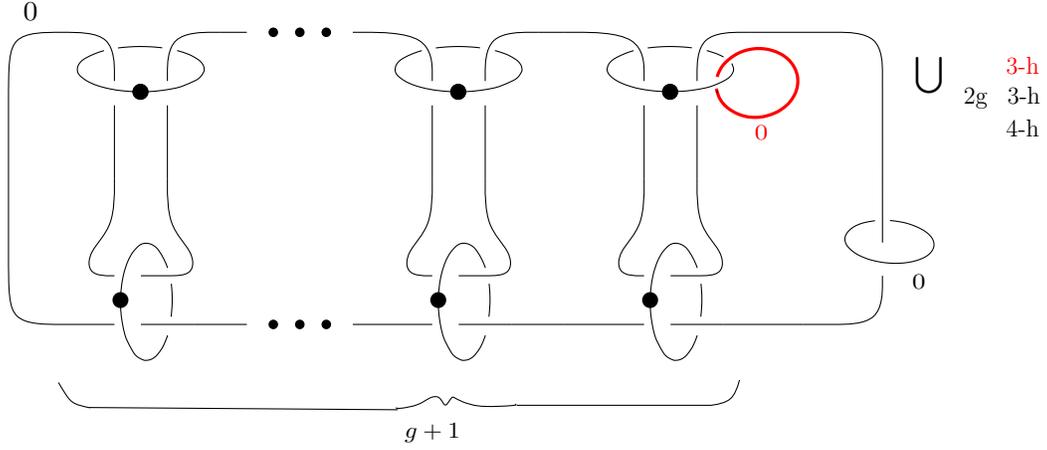}
\caption{\small The step fibration on $S^2 \x \Sigma_g \, \# S^1 \x S^3$.}
\label{f:example1}
\end{center}
\end{figure}

In several aspects, the round handle cobordism $W$ in the step fibration is the simplest possible cobordism. Here $W$ is the product of $S^1$ with a $3$-dimensional cobordism from $\Sigma_{g+1}$ to $\Sigma_g$ given by a $2$-handle attachment. We refer to these type of cobordisms as \emph{elementary cobordisms}. The round handle cobordisms in Example \ref{StandardBLFexample} are \textit{all} elementary.

\begin{figure}[ht]
\begin{center}
\includegraphics[scale=1.2]{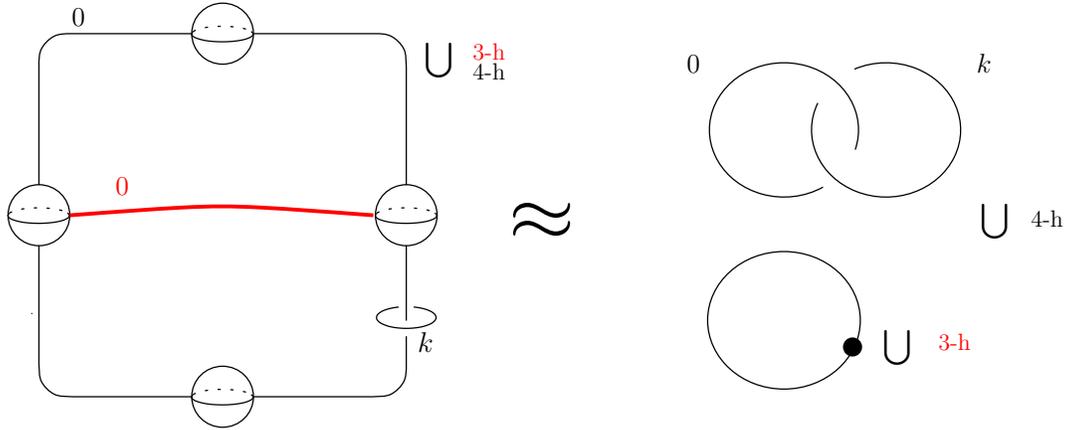}
\caption{\small A family of near-symplectic BLFs over $S^2$ (left), and the diagram after the handle slides and cancellations (right).}
\label{ADKexample}
\end{center}
\end{figure}

When $g=0$ we obtain a more general family as in Figure \ref{ADKexample}, where the section can now assume any self-intersection number $k$ depending on the identification of the lower side boundary. The fibrations we get are precisely the near-symplectic examples of \cite{ADK}. After simple handle slides and cancellations, one gets a diagram of the connected sum of an $S^2$ bundle over $S^2$ with Euler class $k$ and an $S^1 \x S^3$. Thus we get $S^2 \x S^2 \# S^1 \x S^3$ for even $k$ and $S^2 \widetilde{\times} S^2 \# S^1 \x S^3$ for odd $k$.
\end{example}

\begin{example} \label{example2} \rm

In Figure \ref{f:example2} we describe a family of simplified BLFs composed of an elliptic Lefschetz fibration with two critical points on the higher side, a trivial sphere fibration on the lower side, and a twisted round singularity in between. We claim that for even $k$ the total space is $S^2 \x S^2$ and for odd $k$ it is $\CP \# \CPb$. In order to verify this we prefer to use the diagram with dotted notation on the right of the Figure \ref{f:example2}. Let $H_2$ be the $2$-handle of the round $2$-handle, given in red and with fiber framing $0$. Using $H_2$, first unlink all the $2$-handles from the top $1$-handle, and cancel this $1$-handle against $H_2$. Then slide the $+1$-framed $2$-handle over the $-1$-framed $2$-handle to obtain the third diagram in the Figure \ref{example2moves}, and cancel the surviving $1$-handle against the ($-1$)-framed $2$-handle. Finally cancel the remaining unlinked $0$-framed $2$-handle against the $3$-handle. The result follows.

\begin{figure} [ht]
\begin{center}
\includegraphics{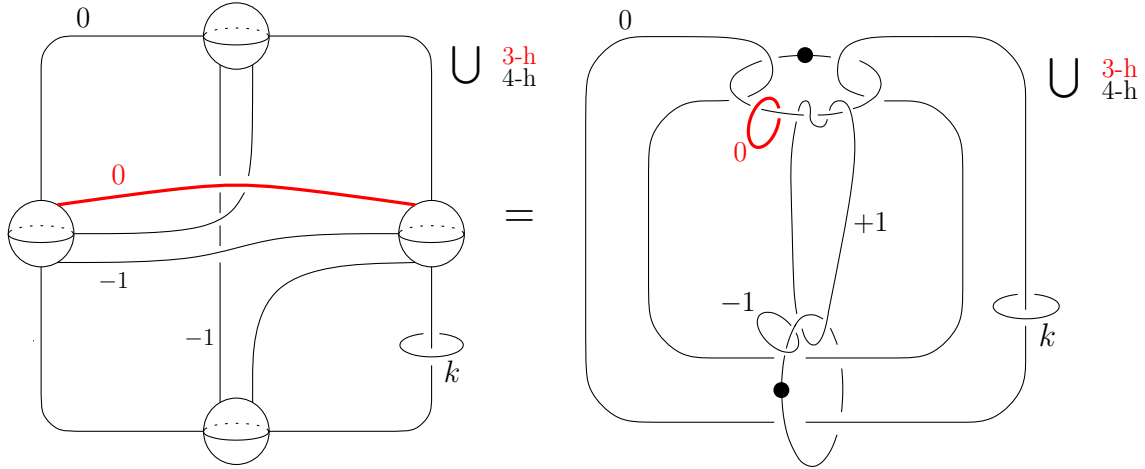}
\caption{\small A family of near-symplectic BLFs with twisted round cobordism. On the right: $1$-handles are replaced by dotted circles.}
\label{f:example2}
\end{center}
\end{figure}

\begin{figure}[ht]
\begin{center}
\includegraphics{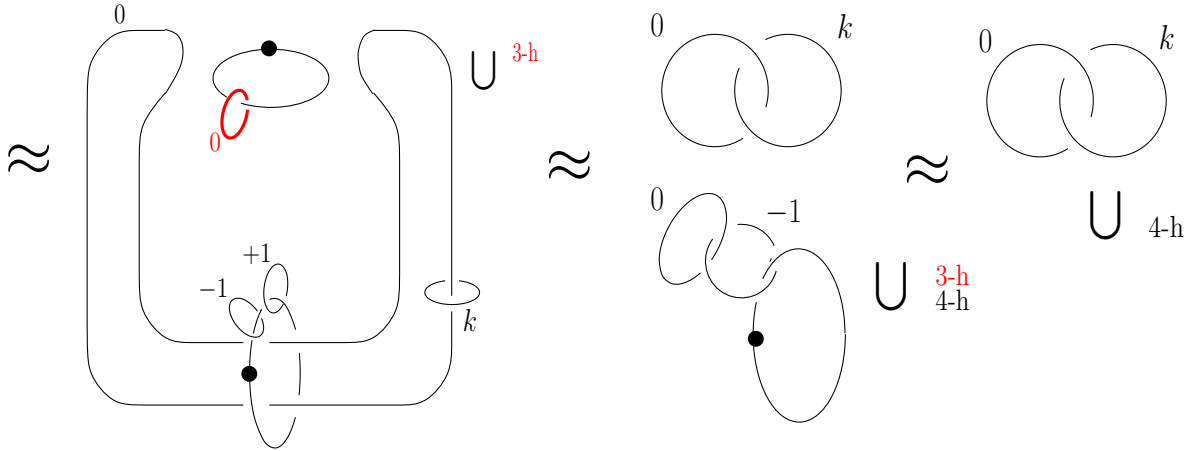}
\caption{\small Identifying the total space of the BLF in Figure \ref{f:example2}.}
\label{example2moves}
\end{center}
\end{figure}

For $k=0$ this is Perutz's Example in \cite{P1}. Moreover, when $k=-1$ the blow-down of this exceptional sphere yields a near-symplectic broken Lefschetz pencil on $\CP$.
\end{example}

\medskip
All the examples we discussed so far had nonseparating round $2$-handles; in other words, in all examples all the fibers were connected. However separating round $2$-handles arise quite naturally when studying broken fibrations on connected sums of $4$-manifolds, as illustrated in the next example.

\begin{example} \label{example3} \rm

Since $b^+( 2 \CP) = 2$, there exists a near-symplectic form on this non-symplectic \, $4$-manifold. We will construct a near-symplectic structure which restricts to a symplectic structure on each $\CP$ summand away from the connected sum region, through BLFs. Take the rational fibrations $f_i$, $i=1,2$ on two copies of $\CP \# \CPb$, with $(-1)$-sections. Consider a fibration $f = f_1 \cup f_2$ on the disjoint union of these two, by simply imagining them `on top of each other'. Now in a regular neighborhood of a fiber of $f$, introduce a round $1$-handle so to connect the disjoint sphere fibers. The result is a BLF $\hat{f}: 2\CP \# 2 \CPb \to S^2$ with two exceptional spheres. Let $h$ be the Poincar\'{e} dual of the sum of $(-1)$-sections. Then $h$ evaluates positively on each fiber component of this fibration, so there exists a near-symplectic structure compatible with $\hat{f}$ with respect to which the two $(-1)$-sections are symplectic. Blowing-down these two sections we obtain a near-symplectic BLF on $2 \CP$ with the proposed properties. A diagram of this fibration is given in Figure \ref{f:example3}.

\begin{figure}
\begin{center}
\includegraphics{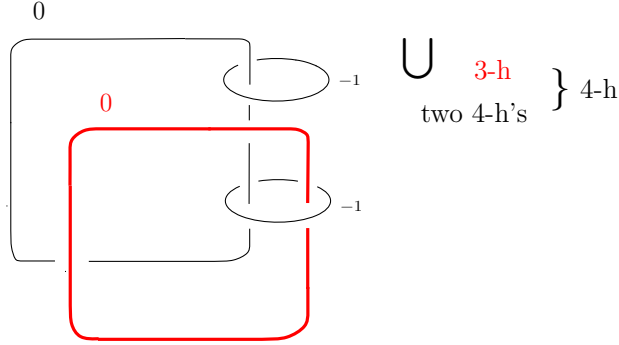}
\caption{\small A near-symplectic BLF on $2 \CP \# 2 \CPb$. The round $2$-handle separates the sphere fiber on the higher side into two spheres on the lower side. }
\label{f:example3}
\end{center}
\end{figure}

\begin{remark} (Broken fibrations on connected sums) \label{BLFonconnectedsum} \rm
The very same idea can be applied to connected sums of any two near-symplectic BLFs over the same base, say by connect summing in the higher genus sides. (This observation is due to Perutz \cite{P2}). Abstractly, for the diagrams of such fibrations over $S^2$, first slide a $2$-handle $F_1$ corresponding to a fiber component over the $2$-handle $F_2$ corresponding to the other fiber component. Then regard $F_2$ as the $2$-handle of a round $2$-handle, and add an extra $3$-handle to the union of two fibration diagrams. This way we obtain a connected sum model for our fibration diagrams.
\end{remark}

\end{example}

\medskip
Using similar techniques, we can also depict diagrams of BLFs which do not necessarily support near-symplectic structures. Next example as well as the family of examples discussed in the following subsection are of this sort:

\begin{example} \label{example4} \rm

As discussed in \cite{ADK}, a modification of $g=0$ case in Example \ref{example1}, yields a BLF on $S^4$. This can be realized by gluing the round cobordism $W$ to the higher side fibration over $D^2$ by twisting the fibration on $\partial_+ W = T^3$ by a loop of diffeomorphisms of the $T^2$ fiber corresponding to a unit translation in the direction transverse to the vanishing cycle $\gamma$ of the round $2$-handle \cite{ADK}. As a result of this, the $2$-handle corresponding to the $S^2$ fiber of the lower side is pulled to the blue curve in Figure \ref{BLFonS4}. The diagram then can be simplified as before: Use the $2$-handle of the round $2$-handle to separate the $2$-handle corresponding to the fiber, and then proceed with the obvious handle cancellations. Introducing a $(-1)$-framed unknot linked with the same $1$-handle that the $2$-handle of the round $2$-handle links in the diagram, we get an honest broken Lefschetz fibration on $\CPb$ (cf. \cite{GK}). 

\begin{figure}
\begin{center}
\includegraphics{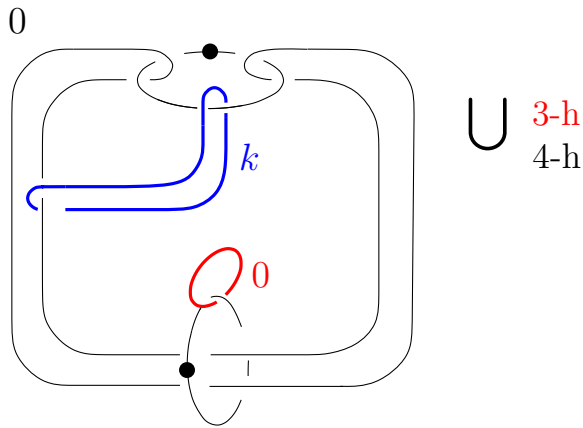}
\caption{\small A broken Lefschetz fibration on $S^4$.}
\label{BLFonS4}
\end{center}
\end{figure}

\medskip
\subsection{Achiral versus broken. } \label{achiralversusbroken} Recall that an \emph{achiral Lefschetz fibration} is defined in the same way a Lefschetz fibration is defined, except that the given charts around critical points are allowed to reverse orientation. An \emph{achiral Lefschetz pencil} is then defined by allowing orientation reversing charts around the base points as well. Critical points with the nonstandard orientation are called \emph{negative critical points}. 

\textit{There are $4$-manifolds which do not admit achiral fibrations or pencils, but admit broken fibrations.} The manifolds $\#_n \, S^1 \x S^3$ do not admit achiral Lefschetz fibrations or pencils when $n \geq 2$ \cite{GS}. Taking the product of the Hopf fibration $S^3 \to S^2$ with $S^1$, we get a fibration $S^1 \x S^3 \to S^2$ with inessential torus fibers. Then the connected sum model discussed in the previous example allows us to construct a broken fibration on any number of connected sums of $S^1 \x S^3$s. In Figure \ref{2S1xS3kirbydiagram} we give a diagram for the $n=2$ case.

\begin{figure}
\begin{center}
\includegraphics{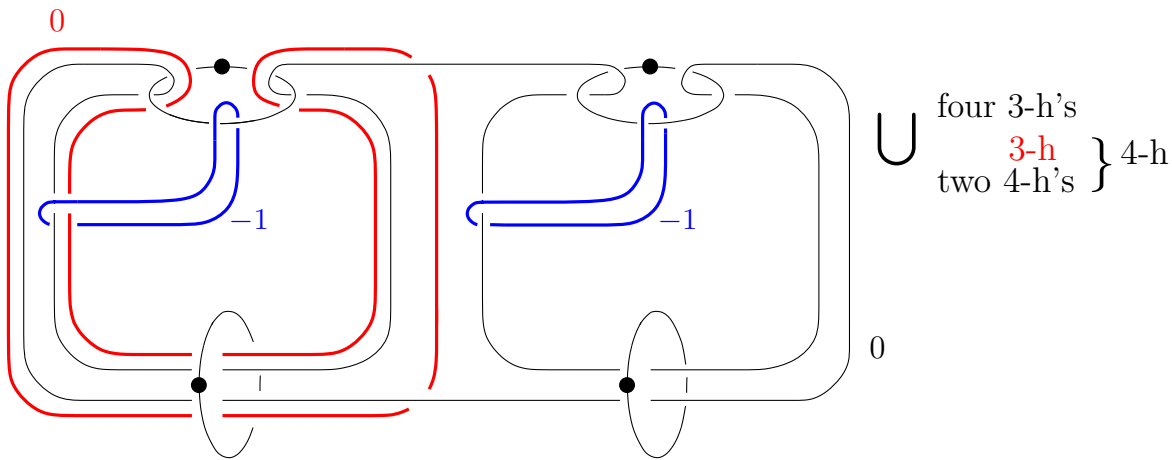}
\caption{\small A broken Lefschetz fibration on $S^1 \x S^3 \, \# \, S^1 \x S^3$.}
\label{2S1xS3kirbydiagram}
\end{center}
\end{figure}
\end{example}
There is a simple local modification around the image of an isolated negative Lefschetz critical point to obtain a new fibration on the blow-up of the $4$-manifold at this critical point, where the singularity is traded with a round singularity. It is equivalent to performing the local operation described at the very end of \cite{ADK} \textit{in an orientation reversing chart on the $4$-manifold}. \footnote{This trick was known to the author for some time, and its proof via handle diagrams given here was contained in his thesis work. We were later informed that Tim Perutz made the same observation in reference to a question of David Gay and Rob Kirby.}

Let $X$ be a compact orientable $4$-manifold, $S$ be a compact orientable surface and $f: X \to S$ be a broken achiral Lefschetz fibration. Assume that $x \in X$ is a negative Lefschetz critical point of $f$. For simplicity, we first assume that there is no other critical point on the fiber that $x$ lies in, and the corresponding vanishing cycle $\gamma$ is a nonseparating curve. Let $V$ be a small disk around $f(x)$ whose intersection with the image of the singular locus of $f$ consist of this point only. It suffices to study our modification in the local model in Figure \ref{NegativeLFCriticalPoint}. This is because there exists a self-diffeomorphism of the fiber which takes $\gamma$ to any nonseparating curve, and it can be extended to a fiber orientation preserving diffeomorphism $\phi$ from $\partial f^{-1}(V)$ to the boundary of the node neighborhood we have. After the modification we glue the new piece back via the diffeomorphism $\phi$ on the boundary which will remain the same throughout the modification. 

\begin{figure} [ht]
\begin{center}
\includegraphics{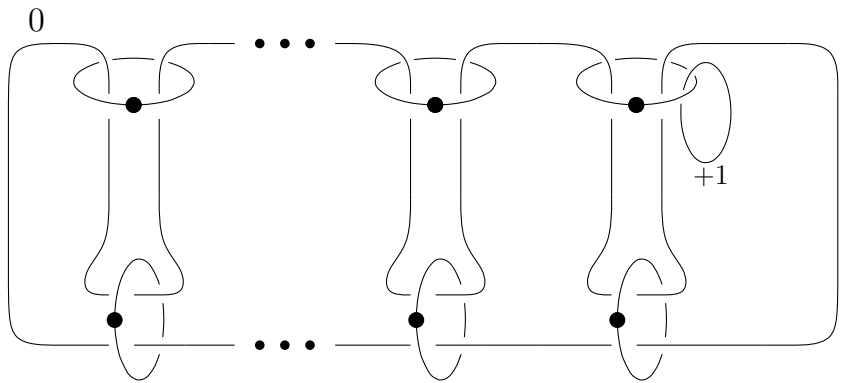}
\caption{\small Neighborhood of a negative nodal fiber with a nonseparating vanishing cycle.} \label{NegativeLFCriticalPoint}
\end{center}
\end{figure}

\begin{figure} 
\begin{center}
\includegraphics{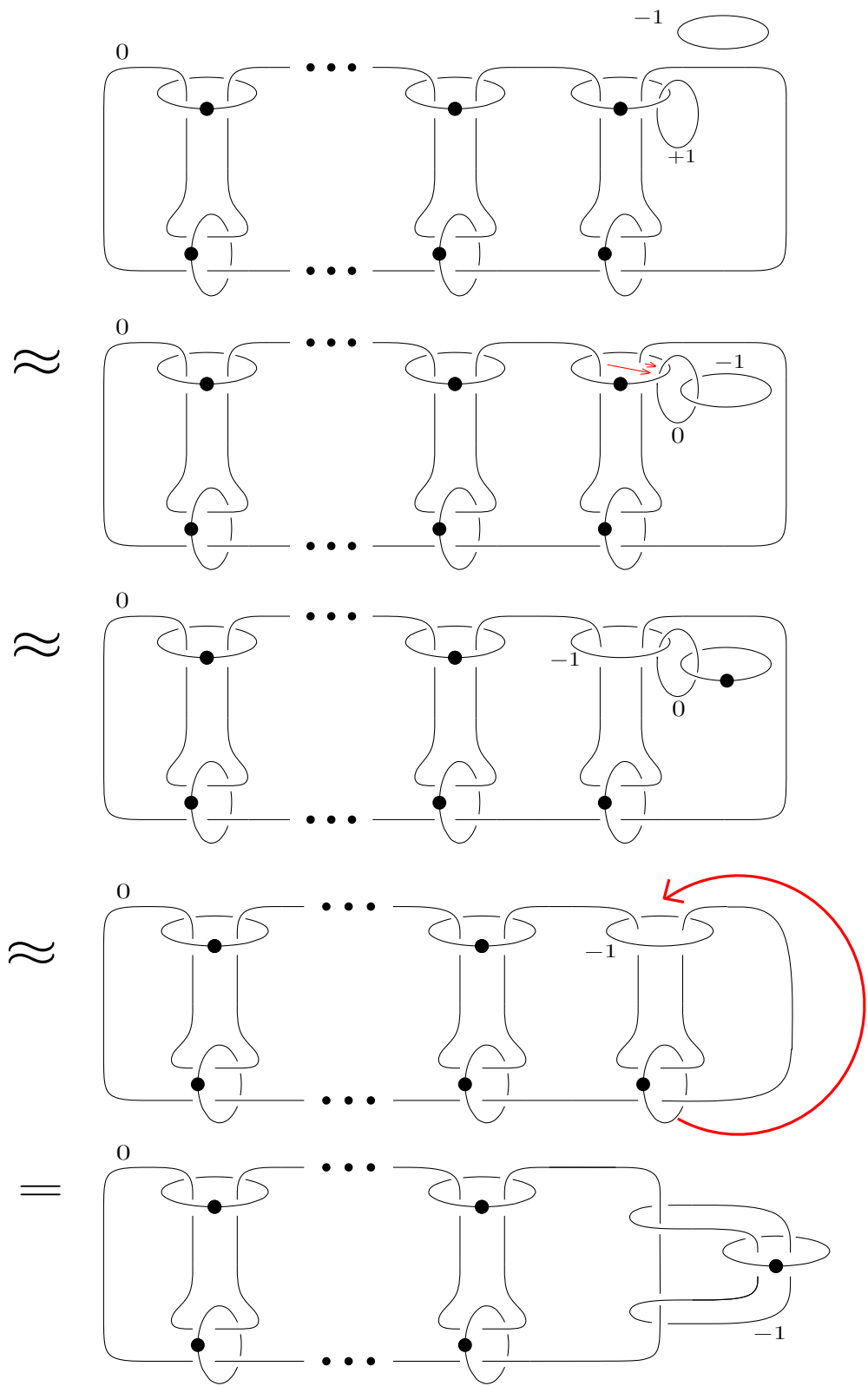}
\caption{\small Consecutive $2$-handle slides in the blown-up neighborhood of a negative node. In the last step, after an isotopy, we obtain a Kirby diagram of a round $1$-handle attachment to a product neighborhood of a fiber with one less genus. }
\label{LocalModification1}
\end{center}
\end{figure}

After blowing-up in this piece, one obtains a new diagram with no Lefschetz singularity but with a new round handle as shown in Figure \ref{LocalModification1}. We first slide the $(+1)$-framed $2$-handle over the $(-1)$-framed $2$-handle so that its framing becomes $0$. Then the two strands of the $0$-framed $2$-handle can be slid off the $1$-handle using the new $0$-framed $2$-handle, and now they go through the $(-1)$-framed $2$-handle as shown in the third diagram. The new $0$-framed $2$-handle and the $1$-handle becomes a canceling pair, which we remove from the diagram. The last step is just an isotopy which puts the diagram in the standard form of a trivial fibration with a fiber of one less genus, and a round $1$-handle attached to it. Observe that the framing of the $2$-handle of the round $1$-handle is $-1$, compensating for the loss of the singular fiber on the boundary monodromy. Lastly note that if there were other Lefschetz critical points on the same fiber then one would have additional $2$-handles for them in the local model, but this would not affect the modification.



\medskip
\subsection{Broken fiber sum. }  \label{BrokenFiberSumSection}

We move on to presenting a generalization of the symplectic fiber sum operation \cite{MW, Go} to the near-symplectic case, which can also be set as a fibered operation.

Let $(X_i, f_i)$ be broken Lefschetz fibrations, and $F_i$ be chosen regular fibers of genus $g_i > 0 $, $i = 1, 2$. Choose regular neighborhoods $N_i=f_i^{-1}(D_i)$ of $F_i$, and without loss of generality, assume $g_1 - g_2 = k$ is a non-negative integer. Then we can obtain a new $4$-manifold $X= X_1 \setminus N_1 \cup W \cup X_2 \setminus N_2$, where $W$ is a composition of $k$ elementary round $2$-handle cobordisms. These cobordisms being elementary implies that the $2$-handles of the round $2$-handles can all be pushed onto a regular fiber $F_1$. The resulting manifold is uniquely determined by an unordered tuple of attaching circles $(\gamma_1, \cdots, \gamma_k)$ of the round $2$-handles involved in $W$, together with the gluing maps $\phi_1 : \partial X_1 \to \partial_+ W$ and $\phi_2 : \partial X_2 \to \partial_- W$ preserving the fibrations. (Recall that these gluings are unique up to isotopy when the fiber genus is at least two.) Hence we obtain a new broken Lefschetz fibration $(X,f)$ that extends the fibrations $(X_i \setminus N_ i, \, f_i|_{X_i \setminus N_i})$ by standard broken fibrations over the elementary cobordisms. We say $(X, f)$ is the \emph{broken fiber sum} of $(X_1, f_1)$ and $(X_2, f_2)$ along $F_1$ and $F_2$, determined by $\gamma_1, \cdots, \gamma_k$ and $\phi_1, \phi_2$.

\begin{theorem} \label{NSbrokensum}
If $(X_i, f_i)$ are near-symplectic broken Lefschetz fibrations, then $(X, f)$ is a near-symplectic broken Lefschetz fibration. Moreover, given arbitrarily small collar neighborhoods $\tilde{N}_i$ of $\partial (N_i)$ in $X_i$, we can choose $\omega$ so that $\omega |_{X_1 \setminus \tilde{N}_1} = \omega_1 |_{X_1 \setminus \tilde{N}_1}$ and $\omega |_{X_2 \setminus \tilde{N_2}} = c \, \omega_2 |_{X_2 \setminus \tilde{N_2}}$, where $c$ is some positive constant.
\end{theorem}

\begin{proof}
Let $k$ be as above. Take step fibrations on $S^2 \x \Sigma_g \, \# S^1 \x S^3$ described in Example \ref{example2} with $g = g_2, \, g_2 +1 , \ldots, g_2 + k = g_1$. Take the fiber sum $S^2 \x \Sigma_{g_2} \, \# S^1 \x S^3$ along a high genus fiber with $S^2 \x \Sigma_{g_2 +1} \, \# S^1 \x S^3$ along a low genus fiber. Then take the fiber sum of this new broken fibration along a high genus fiber with $S^2 \x \Sigma_{g_2+2} \, \# S^1 \x S^3$ along a low genus fiber, and so on, until $g = g_2 +k$. Denote this manifold by $\tilde{W}$. Since the BLF on $\tilde{W}$ admits a section, it can be equipped with a near-symplectic structure. Hence the broken fiber sum of $(X_1, f_1)$ and $(X_2, f_2)$ along $F_1$ and $F_2$ is obtained by fiber summing the former along $F_1$ with $\tilde{W}$ along a lower side fiber, and the latter along $F_2$ with $\tilde{W}$ along a higher side fiber. We can make these fiber sums symplectically, after possibly rescaling one of the near-symplectic forms $\omega_i$, $i=1,2$. When $k=0$ this would be the usual symplectic fiber sum.
\end{proof}

\begin{remark} \label{brokenfibersumdiagram} \rm
If $(X_i, f_i)$ for $i =1,2$ are Lefschetz fibrations over $S^2$, then one can depict the Kirby diagram of the broken fiber sum $(X, f)$ in terms of these two by using Lemma \ref{Push}. Since the round cobordism in the broken fiber sum consists of elementary cobordisms, all the $2$-handles of the round $2$-handles and the Lefschetz handles of the lower genus fibration can be drawn on the higher genus fiber directly.
\end{remark}

\begin{remark} \rm
Forgetting the fibration maps, we can describe the above construction for any two near-symplectic $(X_i, \omega_i)$, containing symplectically embedded surfaces $F_i$ with trivial normal bundles, where $i=1,2$. It is also possible to form a cobordism similar to $W$ when $F_1^2 = - F_2^2 \neq 0$ to handle the most general situation, but we won't have more to say about this here.
\end{remark}

Topological invariants of $X$ are easily determined. For example if $X_i$ are simply-connected and at least one of them admits a section, then using Seifert-Van Kampen theorem we conclude that $X$ is also simply-connected. The Euler characteristic and signature of $X$ can be expressed in terms of those of $X_1$ and $X_2$ as:
\begin{equation}
\eu(X) = \eu(X_1) + \eu(X_2) + 2 (g_1 + g_2) -4 \,\, , \,\,\,\,\,\,\, \sigma(X) = \sigma(X_1) + \sigma(X_2).
\end{equation}

\noindent where $g_i$ is the genus of $F_i$, for $i=1,2$. Therefore the holomorphic Euler characteristic $\chi_h(X) = \chi_h(X_1) + \chi_h(X_2) - 1 - (g_1 + g_2) / 2$. It follows that if $X_1$ and $X_2$ are almost complex manifolds, then $X$ obtained as their broken fiber sum along $F_1$ and $F_2$ is almost complex if and only if $k \equiv g_1 + g_2 \equiv 0 \pmod2$. Lastly note that the broken fiber sum operation might introduce second homology classes in $X$ that do not come from $X_i$ in addition to the usual Rim tori. This phenomenon occurs for instance when some $\gamma_i$ match with relative disks in $X_2 \setminus N_2$ to form an immersed sphere $S_i$. Then the torus $T_i$, which corresponds to a submanifold $\alpha_i \times S^1 \subset \partial (X_2 \setminus N_2) \cong F_2 \x S^1$, where $\alpha_i$ is the dual circle to $\gamma_i$ on $F_2$, intersects with $S_i$ at one point.

\begin{example} \rm
Take $X_1 = S^2 \x T^2 \# 4 \CPb$ with the Matsumoto fibration $f_1 : X_2 \to S^2$, and $X_2 = S^2 \x S^2$ with the trivial rational fibration  $f_2 : X_2 \to S^2$. The former is a genus two fibration and has the global monodromy: $(\beta_{1} \beta_{2} \beta_{3} \beta_{4})^{2} = 1$, where the curves $\beta_{1}$, $\beta_{2}$, $\beta_{3}$ and $\beta_{4}$ are as shown in Figure \ref{Matsumoto}.

\begin{figure}
\begin{center}
\includegraphics[scale=1.1]{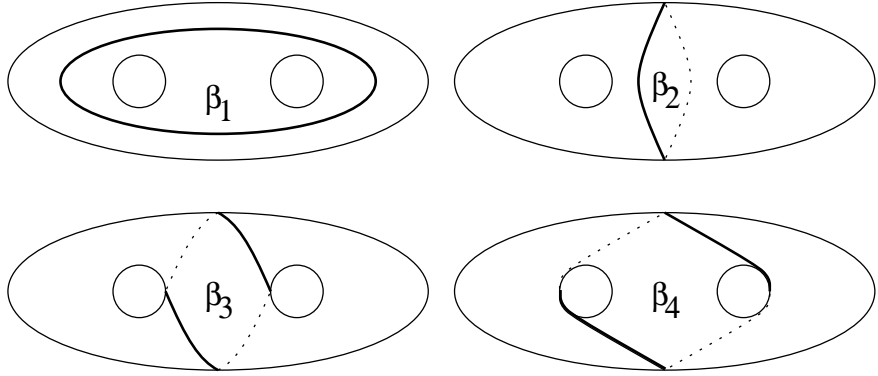}
\caption{\small Vanishing cycles in the Matsumoto fibration.}
\label{Matsumoto}
\end{center}
\end{figure}

If we denote the standard generators of the fundamental group of the regular fiber $\Sigma_2$ as $a_1, b_1, a_2, b_2$, then the curves $\beta_i$ are base point homotopic to:
\noindent $\beta_{1} = b_{1}b_{2}$, \
$\beta_{2} = a_{1}b_{1}{a_{1}}^{-1}{b_{1}}^{-1} = a_{2}b_{2}{a_{2}}^{-1}{b_{2}}^{-1}$, \
$\beta_{3} = b_{2}a_{2}{b_{2}}^{-1}a_{1}$, \
$\beta_{4} = b_{2}a_{2}a_{1}b_{1}$.

\noindent Hence $\pi_1(X_1) = \pi_1(\Sigma_2) \, / \, \langle \beta_{1},\beta_{2},\beta_{3},\beta_{4} \rangle$ is isomorphic to
\begin{eqnarray*}
\pi_1(X_1) &=& \langle a_1,b_1,a_2,b_2 \mid b_1 b_2 = [a_1,b_1]=[a_2,b_2]=b_2 a_2 b_2^{-1} a_1 = 1 \rangle.
\end{eqnarray*}

Now take the broken fiber sum of $(X_1, f_1)$ and $(X_2, f_2)$ along regular fibers $F_1$ and $F_2$, where $\gamma_1 = a_1$, $\gamma_2 = b_2$. The gluing map $\phi_1$ is unique, and we take $\phi_2$ as the identity. Thus we get a new $4$-manifold $X$ and a near-symplectic BLF $f: X \to S^2$ with two untwisted round singular circles.  Note that $\pi_1(X_1 \setminus N(F_1)) \cong \pi_1(X_1)$, and $\pi_1(X_2 \setminus N(F_2)) = 1$, since there are spheres orthogonal to each fiber $F_i$ in $X_i$. From Seifert-Van Kampen's theorem and from the choice of $\gamma_i$ in the broken sum, we see that
\begin{eqnarray*}
\pi_1(X) &=& \langle a_1,b_1,a_2,b_2 \mid b_1 b_2 = [a_1,b_1]=[a_2,b_2]=b_2 a_2 b_2^{-1} a_1= a_1 = b_2 = 1 \rangle.
\end{eqnarray*}
\noindent Thus $\pi_1(X) = 1$. On the other hand, $\eu(X) = \eu(X_1) + \eu(X_2) +  2 (g_1 + g_2) -4 = 8$, and $\sigma(X) = \sigma(X_1) + \sigma(X_2) = -4$. Hence, $X$ is homeomorphic to $\CP \# 5 \CPb$ by Freedman's Theorem. Moreover we obtain four distinct symplectic sections of self-intersection $-1$ in $(X, f)$ which arise from the internal connected sum of four parallel copies of the self-intersection zero section of $S^2 \x S^2 \cup W$ and the four $-1$-sections in the Matsumoto fibration in the broken fiber sum. Symplectically blowing-down these sections, we get a near-symplectic structure with two untwisted round circles on a homotopy $S^2 \x S^2$, together with a broken Lefschetz pencil supporting it. One can indeed verify that the total space is $S^2 \x S^2$ using the Remark \ref{brokenfibersumdiagram}.
\end{example}

What makes the broken fiber sum operation interesting is that, a priori, \textit{gluing formulae can be given for the invariants.} For if we compute the Seiberg-Witten-Floer invariants using the decomposition $X_1 \setminus N_1 \cup W \cup X_2 \setminus N_2$, the work in \cite{D} shows that on $W$ the maps between the relative Floer invariants is standard. That is, if $W$ consists of elementary cobordisms corresponding to $\gamma_j$ by $W_j$, and if Poincar\'{e}-Lefschetz duals of $\gamma_j$ on $F_1$ are $c_j$, then on $W$, this map is given by wedging with $c_j$ under the Piunikhin-Salamon-Schwarz isomorphism (defined for a given $\spinc$) between Floer homologies and singular homology. 

Although in many situations the broken fiber sum of near-symplectic $4$-manifolds can result in $4$-manifolds with vanishing \SW invariants, there are examples when it doesn't:

\begin{example} \label{nontrivialSW} \rm
Let $X_1= S^2 \x \Sigma_{g+1}$ and $X_2= S^2 \x \Sigma_g$ with projections $f_i$ on the $S^2$ components. The broken fiber sum $(X, f)$ of $(X_1, f_1)$ and $(X_2, f_2)$ along the fibers $\Sigma_{g+1}$ and $\Sigma_g$ is the same as $S^2 \x \Sigma_g \, \# S^1 \x S^3$ equipped with the step fibration. This has nontrivial \SW invariants (cf. \cite{OS:adj ineq}), calculated in the Taubes chamber of a compatible near-symplectic form. (Since both $S^2 \x {\text pt}$ and ${\text pt} \x \Sigma_2$ are symplectic with respect to these near-symplectic structures, the near-symplectic forms can be chosen so that they are homologous to the product symplectic form. Therefore \SW invariants are computed nontrivially in the \textit{same} chamber.)
\end{example}

A similar argument can be used to calculate \SW nontrivially, in general for the broken fiber sum of any symplectic Lefschetz fibration $(Y, f)$ of genus $g$ and $b^+(Y) > 1$ with the trivial fibration on $S^2 \x \Sigma_{g+1}$. The same type of handle calculus shows that the resulting manifold is $Y \# S^1 \x S^3$. Since $Y$ has nontrivial \SW, so does $Y \# S^1 \x S^3$ \cite{OS:adj ineq}. Moreover in \cite{OS:adj ineq}, the authors show that the dimension of the moduli space for such a nontrivial solution increases to one, thus $Y \# S^1 \x S^3$ is not of simple type. Having the simple type conjecture for \SW invariants of simply-connected $4$-manifolds with $b^+>1$ in mind, we therefore ask:

\noindent \textbf{Question 1:} Are there near-symplectic $4$-manifolds $X_1$ and $X_2$ with symplectically embedded surfaces $F_1 \hookrightarrow X_1$ and $F_2 \hookrightarrow X_2$ of different genera, such that their broken fiber sum along $F_1$ and $F_2$ results in a \textit{simply connected} $4$-manifold $X$ with nontrivial Seiberg-Witten invariant?

\medskip
\section{Near-symplectic $4$-manifolds with non-trivial Seiberg-Witten invariants} \label{applications}

We now turn our attention to near-symplectic $4$-manifolds whose \SW invariants are nontrivial. Let us refer to these as \emph{nontrivial near-symplectic 4-manifolds} for a shorthand, even though we do not claim that the \SW calculation makes use of the near-symplectic forms. \textit{When $b^+=1$ though we always assume that the \SW is computed in the chamber of the near-symplectic form}. The point of view we take is to regard this as an intermediate class which lies in between near-symplectic and symplectic classes. In this section we will investigate how far one can push certain results regarding symplectic $4$-manifolds and Lefschetz fibrations.

One might wonder if the class of nontrivial near-symplectic 4-manifolds is closed under the symplectic fiber sum operation, as it is the case for both near-symplectic and symplectic $4$-manifolds, respectively. Next theorem not only states that this is too much to hope but also points out how the choice of symplectic surfaces play a significant role here:

\begin{theorem} \label{trivialsum}
There are pairs of closed near-symplectic $4$-manifolds with nontrivial Seiberg-Witten invariants whose symplectic fiber sum along some fibers result in $4$-manifolds with vanishing Seiberg-Witten, whereas along some others they result in $4$-manifolds with nontrivial Seiberg-Witten.
\end{theorem}

\begin{proof}
As discussed in Example \ref{nontrivialSW} and the succeeding paragraph, if $Y$ has nontrivial \SW, then so does $Y \# \, S^1 \x S^3$. Take $E(n)$ (say with $n > 1$) with an elliptic fibration, and equip it with a symplectic form making the regular torus fiber $T$ symplectic. Also take $S^2 \x \Sigma_2$ with the product symplectic form. Look at the broken fiber sum of $E(n)$ with $n \geq 2$ along a regular torus fiber $T$ with $S^2 \x \Sigma_2$ along a genus two surface $\{ pt \} \x \Sigma_2$, where boundary gluings $\phi_1$ and $\phi_2$ are chosen to be identity, and $\gamma$ is chosen to be some fixed standard generator of $\Sigma_2$. The result is the near-symplectic $4$-manifold $X_n \cong E(n) \# \,S^1 \x S^3$, which has nontrivial \SW as noted in the paragraph succeeding the Example \ref{nontrivialSW}.

We can then take the symplectic fiber sum of such $X_n$ and $X_m$ along the higher side genus two fibers to get $X_{n,m}$. There are families of disks with their boundaries on $\partial (X_n \setminus N(\Sigma_2))$ and $\partial (X_m \setminus N(\Sigma_2))$, coming from the broken fiber sum construction in each piece. Matching pairs of these disks give spheres $S_s$ with zero self-intersection, where $s$ is parameterized by the base $S^1$ in the gluing region $S^1 \x \Sigma_2$ of the fiber sum. Denote the equator of $S_s$ sitting on the fiber sum region by $\gamma_s$, and consider a dual circle $\alpha_s$ on the same fiber. Varying $s$ along $S^1$ we obtain a Lagrangian torus $T$, which intersects each $S_s$ at one point. Thus $S_0$ is an essential sphere in $X_{n,m}$. Since $b^+(X_{n,m}) > 1$, the existence of such a sphere implies that $\SW_{X_{n,m}} \equiv 0$. An easy way to see it is as follows: Blow-up at the intersection point of $S_0$ and $T$, and then blow-down the proper transform of $S_0$. This way we get a torus $T'$ with $[T']^2 =1$ in a $4$-manifold with $b^+>1$. The adjunction inequality and the blow-up formula in turn implies that \SW of the original manifold $X_{n, m}$ should have been identically equal to zero. 

However, if one takes the fiber sum along lower genus fibers, the result is $E(n+m) \# \,2\, S^1 \x S^3$, which has nontrivial \SW.  
\end{proof}

\noindent We actually get an infinite family of examples obtained by varying $n, m > 1$ in the proof. The theorem demonstrates that the choice of the fibers in a near-symplectic fiber sum affects the outcome drastically. A natural question that follows is: \\
\noindent \textbf{Question 2:} If $(X_i, f_i)$ are nontrivial $4$-manifolds equipped with near-symplectic broken Lefschetz fibrations and $F_i$ are connected fibers with \textit{with minimal genus}, is the symplectic fiber sum of $X_1$ and $X_2$ along $F_1$ and $F_2$ always nontrivial?

It is known that Lefschetz fibrations over $S^2$ do not admit sections of nonnegative self-intersections, and the self-intersection can be zero only when the fibration is trivial ---i.e when it is the projection from $S^2 \x \Sigma_g$ onto the first component. (See for instance \cite{St}.) In the case of BLFs we see that:

\begin{theorem} \label{sections}
There are closed simply-connected $4$-manifolds which admit near-symplectic broken Lefschetz fibrations over $S^2$ with sections of any self-intersection. More precisely, for any integer $k$ and positive integer $n$, there is a near-symplectic $(X_{n,k} \, , f_{n, k})$ fibered over $S^2$, with a section of self-intersection $k$ and with $b^+(X_{n,k}) = n$. If $f: X \to S^2$ is a nontrivial broken Lefschetz fibration on a nontrivial near-symplectic $4$-manifold $X$, then any section has negative self-intersection if $b^+(X) >1$, but there are examples with sections of any self-intersection when $b^+(X) =1$. 
\end{theorem}

\begin{proof}
In Example \ref{example2} we have constructed near-symplectic BLFs over $S^2$ which admit sections of any self-intersection $k$. As the total space of these fibrations are either $S^2 \x S^2$ or $\CP \# \CPb$, the \SW invariants are nontrivial. (Since the near-symplectic forms can be chosen so that they determine the same chamber as the usual symplectic forms, then \SW invariants are computed nontrivially in their chambers.) These provide examples for the very last part of the theorem. As described in the Example \ref{example3}, we can obtain a near-symplectic BLF on connected sums of these fibrations. Using $n$ such copies, we obtain a $4$-manifold with $b^+ =n$, which proves the first statement. For the remaining assertion, we simply employ the \SW adjunction inequality.
\end{proof}

There are various examples of nonsymplectic $4$-manifolds which have nontrivial \SW invariants. All these examples have $b^+ > 0$, which means that they admit near-symplectic broken Lefschetz pencils but not symplectic Lefschetz fibrations or pencils. This can be made explicit in Fintushel-Stern's knot surgered $E(n)$ examples \cite{FS: knots links, FS: same SW}. Below we obtain near-symplectic BLFs on an infinite family of pairwise nondiffeomorphic closed simply-connected smooth $4$-manifolds which can not be equipped with Lefschetz fibrations or pencils.

\begin{proposition} \label{BLFonknotsurgered}
For any positive integer $n$ and any knot $K$, $E(n)_K$ admits a near-symplectic broken Lefschetz fibration over $S^2$
\end{proposition}

\begin{proof}
Think of $E(n)$ as the branched double cover of $S^2 \times S^2$ with branch set composed of four disjoint parallel copies of $S^2 \times \{pt\}$ and $2n$ disjoint parallel copies of $\{pt\} \times S^2$, equipped with the locally holomorphic `horizontal fibration' \cite{FS: same SW}. The regular torus fiber $F$ of the usual vertical fibration is a bi-section with respect to this fibration. We have exactly four singular fibers each with multiplicity two. On the other hand, if $M_K$ is obtained by a $0$-surgery on a nonfibered knot $K$ in $S^3$, then there is a broken fibration (no Lefschetz singularities) from $S^1 \times M_K$ to $T^2$ as discussed in Example \ref{StandardBLFexample}. One can compose this map with a degree two branched covering map from the base $T^2$ to $S^2$, such that the branching points are not on the images of the round handle singularities. What we get is a broken fibration with four multiple fibers of multiplicity two, which are obtained from collapsing two components from all directions. An original torus section $T$ of $S^1 \x M_K \to T^2$ is now a bi-section of this fibration, intersecting each fiber \textit{component} at one point. Both $F$ and $T$ have self-intersection zero, and thus we can take the symplectic fiber sum of $E(n)$ and $S^1 \x M_K$ along them to get $E(n)_K$. The multiplicity two singular fibers can be matched so to have a locally holomorphic broken fibration with four singular fibers of multiplicity two. This fibration can be perturbed to be Lefschetz as argued in \cite{FS: same SW}. When $K$ is fibered, we obtain genuine Lefschetz fibrations.
\end{proof}

\end{document}